\documentclass[12pt,letterpaper]{amsart}
\usepackage{amsmath, amsthm, amssymb, xspace, mathrsfs}
\usepackage[tmargin=1.4in,bmargin=1.2in,rmargin=1.4in,lmargin=1.4in]{geometry}
\usepackage[breaklinks=true]{hyperref}
\usepackage{amsmath,amscd}
\usepackage{tikz-cd}

\theoremstyle{plain}
\newtheorem{theorem}{Theorem}[section]

\newtheorem{prop}{Proposition}[section]
\newtheorem{cor}{Corollary}[section]

\theoremstyle{definition}

\newtheorem*{acknowledgement}{Acknowledgements}
\newtheorem*{condition}{Condition}

\newcommand{\comment}[1]{}

\begin{document}
\title[Regularity of the Berezin transform]{On regularity of the Berezin transform on smooth pseudoconvex domains}
\author{Akaki Tikaradze}
\email{ tikar06@gmail.com}
\address{University of Toledo, Department of Mathematics \& Statistics, 
Toledo, OH 43606, USA}

\maketitle
\begin{abstract}

In this short note we improve  some of recent results of \v{C}u\v{c}kovi\'c and \c{S}ahuto\u{g}lu \cite{CS}
concerning regularity of the Berezin transform for a class of smooth pseudoconvex domains.
\end{abstract}

\vspace{0.3in}
Let $\Omega\subset\mathbb{C}^n$ be a bounded pseudoconvex domain. As usual A$^2(\Omega)$ denotes
the Bergman space of square integrable holomorphic functions on $\Omega.$
Let $k_w, w\in \Omega$ denote the normalized Bergman reproducing kernel.
Then given a bounded operator $S:A^2(\Omega)\to A^2(\Omega)$ its Berezin transform is defined
as $B(S)(z)=\langle S(k_z), k_z\rangle.$ The Berezin transform has been an important tool in the study of
Toeplitz operators. Recall that given $f\in L^{\infty}(\Omega)$, its Toeplitz operator $T_{f}:A^2(\Omega)\to A^2(\Omega)$
is defined as the composition of the multiplication by $f$ followed by the orthogonal projection $L^2(\Omega)\to A^2(\Omega).$
We need also Hankel operators $H_f=m_f-T_f:A^2(\Omega)\to A^2(\Omega)^{\perp}$, where $m_f$ is the multiplication by $f$.
One defines the Berezin transform of a function $f$ as $B(T_f).$

Hereafter,  $\mathcal{T}(\Omega)$ denotes the algebra generated by all Toeplitz operators with symbols continuous
on $\overline{\Omega}$, and  $\mathcal{K}(\Omega)\subset \mathcal{T}(\Omega)$ denotes the ideal
of compact operators.

 In a recent paper \cite{CS}, \v{C}u\v{c}kovi\'c and \c{S}ahuto\u{g}lu introduces and studied the notion of a BC-rugular domain:
 A domain $\Omega\subset \mathbb{C}^n$ is called BC-regular if for any $S\in \mathcal{T}(\Omega)$, its Berezin
 transform $B(S)$ can be continuously extended on $\overline{\Omega}.$
The authors went to prove that (among other results) a bounded smooth convex domain with no analytic discs in the boundary
is a BC-domain.

To state our results we recall a  well-known fact that if $\partial\Omega$ is smooth, then
for any $w\in\partial\Omega, k_z\to 0$ weakly as $z\to w.$ Therefore, if $S$ is a compact operator then
$B(S)$ vanishes on the boundary of $\Omega.$

It will be convenient to use the following definition.

\begin{condition}
Let  $\Omega$ be a smooth bounded pseudoconvex domain.
Then $B_{\Omega}\subset \partial \Omega$ is defined as the set of all $w\in \partial\Omega$ such that
  for any $f\in C(\overline{\Omega})$ we have $$\lim_{z\to w} B(H^*_fH_f)(z)=0.$$

\end{condition}

It is well-known that all strongly pseudoconvex points belong to $B_{\Omega}.$
It is immediate that if $H_f$ is compact for all $f\in C(\overline{\Omega})$, then $B_{\Omega}=\partial\Omega.$
Recall that if the $\bar{\partial}$-Neumann operator is compact, then all $H_f$ are compact operators for
any $f\in C(\overline{\Omega})$ and hence  $B_{\Omega}=\partial\Omega.$

Next we need to recall the following result of Salinas, Sheu and Upmeier about the maximal commutative quotient of $\mathcal{T}(\Omega).$

\begin{theorem}[\cite{SSU}, Theorem 1.4]\label{S}

Let $I$ denote the commutator ideal in $T(\Omega).$
Assume that $\partial(\bar{\Omega})=\partial\Omega.$ Then there is an isometry of $C^*$ algebras $\eta: T(\Omega)/I\cong C(X),$
 where $X\subset \partial\Omega$ is a closed subset.
If in addition $H_f$ is compact for all $f\in C(\overline{\Omega})$, then $K=I$ and $X=\partial(\Omega).$

\end{theorem}

It follows from the proof that the surjective homomorphism $\eta: T(\Omega)\to C(X)$
is uniquely determined by  the  restriction property $\eta(T_f)=f|_X.$

We show the following result.

\begin{theorem}\label{Main}

Let $ \Omega\subset \mathbb{C}^n$ be a smooth bounded pseudoconvex domain.
Then for any $S\in\mathcal{T}(\Omega)$, its Berezin transform  $B(S)$ can be continuously extended to $\Omega\cup B_{\Omega}$ and the Berezin transform induces a surjective homomorphism on Banach algebras
$ B:\mathcal{T}(\Omega)/\mathcal{K}\to C(B_{\Omega})$ so that $B(T_f)=f|_{B_{\Omega}}.$
 If moreover, $H_f$ is compact for all $f\in C(\overline{\Omega}),$ then the above homomorphism
is an isometry of $C^*$-algebras.

\end{theorem}
\begin{proof}

At first, recall the following relation between semi-commutators of Toeplitz operators and Hankel operators 
 $$T_{f\bar{g}}-T_fT_{\bar{g}}=H^*_{\bar{f}}H_{\bar{g}}, \quad f, g \in C(\overline{\Omega}).$$
Hence, it follows from our assumptions that 
$$B(T_fT_{\bar{g}})|_{B_\Omega}=B(T_{f\bar{g}})|_{B_\Omega}.$$
 Let $f, g\in \mathbb{C}[z],$ then $B(T_fT_{\bar{g}})=f\bar{g}.$
 Therefore, for any $\phi\in \mathbb{C}[z, \bar{z}]$ we have 
$B(T_{\phi})|_{B_{\Omega}}=\phi|_{B_{\omega}}.$ Hence, using  the Stone-Weierstass theorem we get
$$B(T_\psi)|_{B_{\Omega}}=\psi_{B_{\Omega}}, \quad\psi\in C(\overline{\Omega}).$$
  Combining this with the above formulas, we conclude that
for any $\phi_1, \cdots, \phi_m\in C(\overline{\Omega})$, we have 
$$B(T_{\phi_1}\cdots T_{\phi_m})|_{B_{\omega}}=( \phi_1\cdots \phi_m)|_{B_{\Omega}}.$$
Thus, we have a continuous algebra homomorphism 
$$B: \mathcal{T}(\Omega)/K\to C(B_{\Omega})$$ such that $B(T_f)=f|_{B_{\Omega}}.$

If $H_f$ is compact for all $f\in C(\overline{\Omega}),$ then $B_{\Omega}=\partial\Omega$
and our homomorphisms  coincides with the one from Theorem \ref{S}. In particular, it
is an isometry of $C^*$-algebras.

\end{proof}

\begin{cor}
Let $\Omega\subset \mathbb{C}^n$  a smooth bounded pseoudoconvex domain such that
$H_f$ is compact for all $f\in C(\overline{\Omega}).$ Then
$\Omega$ is BC-regular and the essential norm of any $S\in \mathcal{T}(\Omega)$ equals to $L^{\infty}(\partial\Omega)$
norm of $B(S)|_{\partial{\Omega}}.$
\end{cor}

The above corollary generalizes theorems 1, 4, and part of theorem 5 from \cite{CS}. 

As shown in [\cite{CS}, Theorem 3] if $\Omega$ is a convex domain with a disc in the boundary and dense strongly pseudoconcex points,
then $\Omega$ is not BC-regular. The following simple result shows a dichotomy for domains with dense strongly pseudoconvex points.

\begin{prop}

Let $\Omega$ be a smooth pseudoconvex domain. Then $B_{\Omega}=\partial\Omega$ if and only if $\Omega$ is BC-regular and
the map $S\to B(S)|_{\partial\Omega}, S\in \mathcal{T}$ is multiplicative. 

Suppose that $B_{\Omega}$ is dense in $\partial\Omega$ (in particular this is the case if
strongly pseudoconvex points are dense in $\partial\Omega.$)
Then $\Omega$ is BC-regular if and only if  $B_{\Omega}=\partial\Omega.$

\end{prop}
\begin{proof}

Suppose that $B: \mathcal{T}(\Omega)/\mathcal{K}\to C(\partial \Omega)$ is multiplicative. Then $B([T_{f}, T_{\bar{f}}])=0$ for a
holomorphic $f.$ Which implies that $B(H^*_{\bar{f}}H_{\bar{f}})=0.$ So, $B_{\Omega}=\partial\Omega$ 
as desired.

If $B_{\Omega}=\partial\Omega$ then Theorem \ref{Main} implies that $\Omega$ is BC-regular.

Now, suppose that $\Omega$ is BC-regular and $B_{\Omega}$ is dense is $\partial\Omega.$ Then for any  $f\in C(\overline{\Omega})$ we have that
$B(f)|_{B_{\Omega}}=f|_{B_\Omega}.$ Since by the assumption $B(f)$ is continuous up to the boundary, we get that
$B(f)=f.$ Now it follows that 
$$B(H_f^*H_f)=B(|f|^2)-|f|^2$$ vanishes on the boundary for all anti-holomorphic $f.$
Hence $\partial\Omega=B_{\Omega}.$

\end{proof}

In view of the above result, it would be interesting to know an example of a smooth  BC-regular domain $\Omega$ for which the $\bar{\partial}$-Neumann
operator is not compact. If the smoothness assumption is dropped, then a polydisc is an example of such a domain [\cite{CS}, Corollary 1].

\begin{acknowledgement}
I am grateful to S.\c{S}ahuto\u{g}lu for explaining results in \cite{CS} and providing  helpful comments.
 Essentially all results (with their proofs) in this note
were suggested by T. Le. 

\end{acknowledgement}

\end{document}